\documentclass[sn-mathphys,Numbered]{sn-jnl}

\usepackage{graphicx}%
\usepackage{multirow}%
\usepackage{amsmath,amssymb,amsfonts}%
\usepackage{amsthm}%
\usepackage{mathrsfs}%
\usepackage[title]{appendix}%
\usepackage{xcolor}%
\usepackage{textcomp}%
\usepackage{manyfoot}%
\usepackage{booktabs}%
\usepackage{algorithm}%
\usepackage{algorithmicx}%
\usepackage{algpseudocode}%
\usepackage{listings}%

\theoremstyle{thmstyleone}%
\newtheorem{theorem}{Theorem}
\theoremstyle{thmstyletwo}%

\theoremstyle{thmstylethree}%
\newtheorem{definition}{Definition}%

\begin{document}
	
	\title[Continuous generalized atomic subspaces for operators]{Continuous generalized atomic subspaces for operators in Hilbert spaces}
	

	\author*[1]{\fnm{Mohamed} \sur{Rossafi}}\email{rossafimohamed@gmail.com; mohamed.rossafi1@uit.ac.ma}
	\equalcont{These authors contributed equally to this work.}
	
		\author[2]{\fnm{Fakhr-dine} \sur{Nhari}}\email{nharidoc@gmail.com}
	\equalcont{These authors contributed equally to this work.}
	
	\author[3]{\fnm{Abdeslam} \sur{Touri}}\email{touri.abdo68@gmail.com}
	\equalcont{These authors contributed equally to this work.}

	\affil[1]{\orgdiv{Laboratory Partial Differential Equations, Spectral Algebra and Geometry}, \orgname{Higher School of Education and Training, University Ibn Tofail}, \orgaddress{\city{Kenitra}, \country{Morocco}}}
	
	\affil[2]{\orgdiv{Laboratory Analysis, Geometry and Applications, Department of Mathematics}, \orgname{Faculty of Sciences, University Ibn Tofail}, \orgaddress{\city{Kenitra}, \country{Morocco}}}
	
	\affil[3]{\orgdiv{Department of Mathematics}, \orgname{Faculty of Sciences, University Ibn Tofail}, \orgaddress{\city{Kenitra}, \country{Morocco}}}

	\abstract{	In this paper, we introduce the concept of continuous $g-$atomic subspace for a bounded linear operator and gives several useful continuous resolution of the identity operator on a Hilbert space by implies  the theory of continuous $g-$fusion frames. Moreover, we introduce the concept of continuous frame operator for a pair of continuous $g-$fusion bessel sequences.}


	
	\maketitle
	
\section{Introduction}

Discret frames were introduced by Duffin and Schaeffer in 1952 \cite{Duf} for studing some profound problems in nonharmonic Fourier series, after the fundamental paper \cite{DGM} by Daubechies, Grossman and Meyer, frame theory began to be widely used, particularly in the more specialized context of wavelet frames and Gabor frames.

The concept of generalization of frames was proposed by G. Kaiser \cite{KAI} and independently by Ali, Antoine and Gazeau \cite{GAZ} to a family indexed by some locally compact space endowed with a Radon measure. These frames are known as continuous frames. Gabrado and Han in \cite{GHN} called these frames Frames associated with measurable spaces, Askari-Hemmat, Dehghan and Radjabalipour in \cite{AHDR} called them generalized frames and in mathematical physics they are referred to as Coherent states \cite{COH}.

A continuous g-frame (or simply a c-g-frame) was firstly introduced by Abdollahpour and Faroughi in \cite{AF} and is an extension of g-frames and continuous frames. Recently, continuous g-frames in Hilbert spaces have been studied intensively. 

Fusion frames were considered by Casazza, Kutyniok and Li in connection with distributed processing and are related to the construction of global frames \cite{CCO}. The fusion
frame theory is in fact more delicate due to complicated relations between the structure
of the sequence of weighted subspaces and the local frames in the subspaces and also due
to the extreme sensitivity to changes of the weights.

Recently, E.Alizadeh, A. Rahimi, E. Osgooei, M. Rahmani have presented continuous $K-g-$fusion frame in  \cite{EAEM}. For more detailed information on frame theory, readers are recommended to consult: \cite{Assila, Ghiati, Karara, Massit, Rossafi, Kabbaj, Chouchene}.

Throught this paper, $U$, $(\Theta, \mu)$ and $\{U_{w}\}_{w\in\Theta}$ will be a separable Hilbert space, a measure space with positive measure $\mu$ and a family of Hilbert spaces, respectively, $\pi_{V}$ is the orthogonal projection from $U$ onto a closed subspace $V$ and $B(U,U_{w})$ will be denoted by $B(U)$. Also, $\mathbb{U}$ will be the collection of all closed subspaces of $U$, and $v:\Theta\rightarrow \mathbb{R}^{+}$ is a measurable mapping such that $v\neq 0$ a.e.

We begin with a few preliminaries that will be needed.
\begin{theorem}\cite{DOG}\label{lem1}
	Let $L,T\in B(U)$. Then the following conditions are equivalent:
	\begin{itemize}
		\item[(1)] $\mathcal{R}(L)\subseteq \mathcal{R}(T)$.
		\item[(2)] $LL^{\ast}\leq \lambda^{2}TT^{\ast}$, for some $\lambda>0$.
		\item[(3)] $L=TS$ for some bounded linear operator $S$ on $U$.
	\end{itemize}
\end{theorem}
\begin{theorem}\cite{KJP}
	The set $\mathcal{S}(U)$ of all self-adjoint operators on $U$ is a partially ordered set with respect to the partial order $\leq$ which is defined as for $T,S\in\mathcal{S}(U)$
	\begin{equation*}
		T\leq S\Longleftrightarrow \langle Tf,f\rangle \leq \langle Sf,f\rangle, \quad \forall f\in U .
	\end{equation*}
\end{theorem}
\begin{theorem}\cite{PGT}\label{lem4}
	Let $V\subset U$ be a closed subspace and $T\in B(U)$. Then 
	\begin{equation*}
		\pi_{V}T^{\ast}=\pi_{V}T^{\ast}\pi_{\overline{TV}}.
	\end{equation*}
	If $T$ is a unitary operator, $i.e. \quad T^{\ast}T=I_{H}$, then $\pi_{\overline{TV}}T=T\pi_{V}$.
\end{theorem}
\begin{definition}\cite{EAEM}
	Let $F:\Omega \rightarrow \mathbb{U}$ be such that for each $f\in U$, the mapping $w\mapsto \pi_{F(w)}f$ is measurable and let $\{U_{w}\}_{w\in\Theta}$ be a collection of Hilbert spaces for each $w\in \Theta$. Suppose that $\chi_{w}\in B(F(w), U_{w})$ and put $\chi=\{\chi_{w}\in B(F(w),U_{w}),w\in\Theta\}$. Then $\{F(w), \chi_{w},v(w)\}_{w\in\Theta}$ is a continuous $g-$fusion frame for $U$ if there exist $0<A\leq B<\infty$ such that for all $f\in U$
	\begin{equation*}
		A\|f\|^{2}\leq \int_{\Theta}v^{2}(w)\|\chi_{w}\pi_{F(w)}f\|^{2}d\mu(w)\leq B\|f\|^{2}.
	\end{equation*}
	where $\pi_{F(w)}$ is the orthogonal projection onto the subspace $F(w)$.
	
	$\{F(w),\chi_{w},v(w)\}_{w\in\Theta}$ is called a tight continuous $g-$fusion frame for $U$ if $A=B$, and Parseval if $A=B=1$. $\{F(w),\chi_{w},v(w)\}_{w\in\Theta}$ is called a Bessel continuous $g-$fusion for $U$ if we only have the upper bound.
	
	Let $U_{0}=\oplus_{w\in\Theta}U_{w}$ and $L^{2}(\Theta, U_{0})$ be a collection of all measurable functions $\phi :\Theta\rightarrow U_{0}$ such that for each $w\in\Theta$, $\phi(w)\in U_{w}$ and $\int_{\Theta}\|\phi(w)\|^{2}d\mu(w)<\infty$.
	
	It can be verified that $L^{2}(\Omega, U_{0})$ is a Hilbert space with inner product defined by 
	\begin{equation*}
		\langle \phi, \psi \rangle =\int_{\Theta}\langle \phi(w),\psi(w)\rangle d\mu(w)
	\end{equation*}
	for $\phi, \psi \in L^{2}(\Theta, U_{0})$ and the representation space in this setting is $L^{2}(\Theta, U_{0})$.
\end{definition}

\begin{definition}\cite{EAEM}
	Let $F: \Theta\rightarrow \mathbb{U}$ be such that for each $f\in U$, the mapping $w\mapsto \pi_{F(w)}f$ is weakly measurable, $K\in B(U)$ and let 
	\begin{equation*}
		\chi=\{\chi_{w}\in B(F(w), U_{0}):w\in\Theta\}.	
	\end{equation*}
	Then $\{F(w),\chi_{w},v(w)\}_{w\in\Theta}$ is a continuous $K-g-$fusion frame for $U$, if there exist $0<A\leq B<\infty$ such that for all $f\in U$.
	\begin{equation}\label{eq1}
		A\|K^{\ast}f\|^{2}\leq \int_{\Theta}v^{2}(w)\|\chi_{w}\pi_{F(w)}f\|^{2}d\mu(w)\leq B\|f\|^{2}
	\end{equation}
	where $\pi_{F(w)}$ is the orthogonal projection of $U$ onto the subspace $F(w)$.
\end{definition}
$\{F(w),\chi_{w},v(w)\}_{w\in\Theta}$ is called a tight continuous $K-g-$fusion frame for $U$ if $A=B$, and Parseval if $A=B=1$. $\{F(w),\chi_{w}, v(w)\}_{w\in\Theta}$ is called a Bessel continuous $g-$fusion for $U$ if the right-hand inequality in \eqref{eq1} holds.

Since each continuous $K-g-$fusion frame is continuous $g-$fusion Bessel, so the synthesis analysis and continuous $K-g-$fusion frame operators are defined. Inded, the synthesis operator is defined weakly as follows
\begin{equation*}
	T:L^{2}(\Omega, U_{0})\rightarrow U
\end{equation*}
\begin{equation*}
	\langle T(\phi),f\rangle =\int_{\Theta}v(w)\langle \pi_{F(w)}\chi_{w}^{\ast}(\phi(w)),f\rangle d\mu(w)
\end{equation*}
where $\phi\in L^{2}(\Theta, U_{0})$ and $f\in U$. It is obvious that $T$ bounded linear operator. Its adjoint, that is called  analysis operator 
\begin{equation*}
	T^{\ast}: U\rightarrow L^{2}(\Theta, U_{0})
\end{equation*}
\begin{equation*}
	T^{\ast}(f)(w)=v(w)\chi_{w}\pi_{F(w)}f,\quad f\in U
\end{equation*}
\begin{definition}\cite{EAEM}
	Suppose that $\{F(w),\chi_{w},v(w)\}_{w\in\Theta}$ is a continuous $K-g-$fusion frame bounds $A$ and $B$. We define $S:U\rightarrow U$ by 
	\begin{equation*}
		\langle Sf,g\rangle =\int_{\Theta}v^{2}(w)\langle \pi_{F(w)}\chi_{w}^{\ast}\chi_{w}\pi_{F(w)}f,g\rangle d\mu(w)
	\end{equation*}
	and we call it the continuous $K-g-$fusion frame operator.
\end{definition}
\begin{theorem}\cite{EAEM}\label{lem3}
	Let $\{F(w),\chi_{w},v(w)\}_{w\in\Theta}$ be a continuous $g-$fusion  Bessel  for $U$. then $\{F(w),\chi_{w},v(w)\}_{w\in\Theta}$ is a continuous $K-g-$fusion frame for $U$ if and only if there exists $A>0$ such that $S\geq AKK^{\ast}$ where $S$ is continuous $K-g-$fusion frame operator.
\end{theorem}

\section{Continuous resolution of the identity operator in continuous $g-$fusion frame}
\begin{theorem}
	Let $\{F(w),\chi_{w},v(w)\}_{w\in\Theta}$ be a continuous $g-$fusion frame for $U$ with bounds $C,D$ and $S_{\Lambda}$ be its associated continuous $f-$fusion frame operator. Then the family $\{v^{2}(w)\pi_{F(w)}\chi_{w}^{\ast}T_{w}\}_{w\in\Theta}$ is the continuous of the identity operator on $U$, where $T_{w}=\chi_{w}\pi_{F(w)}S^{-1}_{\chi}$, $\forall w\in\Theta$. Furthermore, for all $f\in U$ we have, 
	\begin{equation*}
		\frac{C}{D^{2}}\|f\|^{2}\leq \int_{\Theta}v^{2}(w)\|T_{w}f\|^{2}d\mu(w)\leq \frac{D}{C^{2}}\|f\|^{2}.
	\end{equation*}
\end{theorem}
\begin{proof}
	We have $S_{\chi}$ is invertible, then for each $f\in U$
	\begin{align*}
		f=S_{\chi}S_{\chi}^{-1}f&=\int_{\Theta}v^{2}(w)\pi_{F(w)}\chi_{w}^{\ast}\chi_{w}\pi_{F(w)}S_{\chi}^{-1}fd\mu(w)\\&=\int_{\Theta}v^{2}(w)\pi_{F(w)}\chi_{w}^{\ast}T_{w}d\mu(w).
	\end{align*}
	Therefore, $\{v^{2}(w)\pi_{F(w)}\chi_{w}^{\ast}T_{w}\}_{w\in\Theta}$ is a continuous resolution of the identity operator on $U$.
	
	Moreover, for each $f\in U$ we have 
	\begin{align*}
		\int_{\Theta}v^{2}(w)\|T_{w}(f)\|^{2}d\mu(w)&=\int_{\Theta}v^{2}(w)\|\chi_{w}\pi_{F(w)}S_{\chi}^{-1}f\|^{2}d\mu(w)\\&\leq D\|S_{\chi}^{-1}f\|^{2}\\&\leq D\|S_{\chi}^{-1}\|^{2}\|f\|^{2}\\&\leq \frac{D}{C^{2}}\|f\|^{2}.
	\end{align*}
	On the other hand 
	\begin{equation*}
		C\|S_{\chi}^{-1}f\|^{2}\leq \int_{\Theta}v^{2}(w)\|T_{w}f\|^{2}d\mu(w),
	\end{equation*}
	since for each $f\in U$
	\begin{equation*}
		\|f\|^{2}=\|S_{\chi}S_{\chi}^{-1}f\|^{2}\leq \|S_{\chi}\|^{2}\|S_{\chi}^{-1}f\|^{2},
	\end{equation*}
	then 
	\begin{equation*}
		\|S_{\chi}\|^{-2}\|f\|^{2}\leq \|S_{\chi}^{-1}f\|^{2},
	\end{equation*}
	therefore 
	\begin{equation*}
		C\|S_{\chi}\|^{-2}\|f\|^{2}\leq \int_{\Theta}v^{2}(w)\|T_{w}f\|^{2}d\mu(w),
	\end{equation*}
	hence
	\begin{equation*}
		\frac{C}{D^{2}}\|f\|^{2}\leq \int_{\Theta}v^{2}(w)\|T_{w}f\|^{2}d\mu(w).
	\end{equation*}
\end{proof}

\begin{theorem}\label{th1}
	Let $\{F(w),\chi_{w},v(w)\}_{w\in\Theta}$ be a continuous $g-$fusion frame for $U$ with bounds $C,D$ and let $T_{w}:U\rightarrow U_{w}$ be a bounded operator such that $\{v^{2}(w)\pi_{F(w)}\chi_{w}^{\ast}T_{w}\}_{w\in\Theta}$ is a continuous resolution of the identity operator on $U$. Then 
	\begin{equation*}
		\frac{1}{D}\bigg|\bigg|\int_{\Theta}v^{2}(w)\pi_{F(w)}\chi_{w}^{\ast}T_{w}fd\mu(w)\bigg|\bigg|^{2}\leq \int_{\Theta}v^{2}(w)\|T_{w}f\|^{2}d\mu(w),\quad \forall f\in U.
	\end{equation*}		
\end{theorem}
\begin{proof}
	Let $f\in U$ and set $g=\int_{\Theta}v^{2}(w)\pi_{F(w)}\chi_{w}^{\ast}T_{w}fd\mu(w)$. Then 
	\begin{align*}
		\|g\|^{4}&=\langle g,g\rangle^{2}=\langle g,\int_{\Theta}v^{2}(w)\pi_{F(w)}\chi_{w}^{\ast}T_{w}fd\mu(w)\rangle^{2}\\&=\biggl(\int_{\Theta}v(w)\langle \chi_{w}\pi_{F(w)}g,v(w)T_{w}f\rangle d\mu(w)\biggr)^{2}\\&\leq \int_{\Theta}v^{2}(w)\|\chi_{w}\pi_{F(w)}g\|^{2}d\mu(w)\int_{\Theta}v^{2}(w)\|T_{w}f\|^{2}d\mu(w)\\&\leq D\|g\|^{2}\int_{\Theta}v^{2}(w)\|T_{w}f\|^{2}d\mu(w),
	\end{align*}
	so
	\begin{equation*}
		\frac{1}{D}\|g\|^{2}\leq \int_{\Theta}v^{2}(w)\|T_{w}f\|^{2}d\mu(w),
	\end{equation*}
	hence 
	\begin{equation*}
		\frac{1}{D}\bigg|\bigg|\int_{\Theta}v^{2}(w)\pi_{F(w)}\chi_{w}^{\ast}T_{w}fd\mu(w)\bigg|\bigg|\leq \int_{\Theta}v^{2}(w)\|T_{w}f\|^{2}d\mu(w).
	\end{equation*}
\end{proof}	
\begin{theorem}
	Let $\{F(w),\chi_{w},v(w)\}_{w\in\Theta}$ be a continuous $g-$fusion frame with bounds $C,D$ and let $T_{w}:U\rightarrow U_{w}$ be a bounded operator such that $\{v^{2}(w)\pi_{F(w)}\chi_{w}^{\ast}T_{w}\}_{w\in\Theta}$ is a continuous resolution of the identity operator on $U$. If $T_{w}^{\ast}\chi_{w}\pi_{F(w)}=T_{w}$, then 
	\begin{equation*}
		\frac{1}{D}\|f\|^{2}\leq \int_{\Theta}v^{2}(w)\|T_{w}f\|^{2}d\mu(w)\leq DE\|f\|^{2},\quad \forall f\in U.
	\end{equation*}
	Where $E=\sup_{w\in\Theta}\|T_{w}\|^{2}<\infty$.
\end{theorem}	
\begin{proof}
	Because $\{v^{2}(w)\pi_{F(w)}\chi_{w}^{\ast}T_{w}\}_{w\in\Theta}$ is a continuous resolution of the identity on $U$, then  
	\begin{equation*}
		f=\int_{\Theta}v^{2}(w)\pi_{F(w)}\chi_{w}^{\ast}T_{w}fd\mu(w),\quad\forall f\in U.
	\end{equation*}
	For each $f\in U$, using theorem \eqref{th1}, we have 
	\begin{align*}
		\frac{1}{D}\|f\|^{2}&=\frac{1}{D}\bigg|\bigg|\int_{\Theta}v^{2}(w)\pi_{F(w)}\chi_{w}^{\ast}T_{w}fd\mu(w)\bigg|\bigg|\leq \int_{\Theta}v^{2}(w)\|T_{w}f\|^{2}d\mu(w)\\&=\int_{\Theta}v^{2}(w)\|T^{\ast}_{w}\chi_{w}\pi_{F(w)}f\|^{2}d\mu(w)\\&\leq \int_{\Theta}v^{2}(w)\|T_{w}\|^{2}\|\chi_{w}\pi_{F(w)}f\|^{2}d\mu(w)\\&\leq E\int_{\Theta}v^{2}(w)\|\chi_{w}\pi_{F(w)}f\|^{2}d\mu(w)\\&\leq ED\|f\|^{2}.
	\end{align*}
\end{proof}
\begin{theorem}
	$\{F(w),\chi_{w},v(w)\}_{w\in\Theta}$ is a continuous $g-$fusion frame for $U$ if the following conditions hold:
	\begin{itemize}
		\item[(1)] For all $f\in U$ there exists $A>0$ such that 
		\begin{equation*}
			\int_{\Theta}\|\chi_{w}\pi_{F(w)}f\|^{2}d\mu(w)\leq \frac{1}{A}\|f\|^{2}.
		\end{equation*}
		\item[(2)] $\{v(w)\pi_{F(w)}\chi_{w}^{\ast}\chi_{w}\pi_{F(w)}\}_{w\in\Theta}$ is a continuous resolution of the identity operator on $U$. 
	\end{itemize}
\end{theorem}
\begin{proof}
	We have $\{v(w)\pi_{F(w)}\chi_{w}^{\ast}\chi_{w}\pi_{F(w)}\}_{w\in\Theta}$ is a continuous resolution of the identity operator on $U$, then  for each $f\in U$ 
	\begin{equation*}
		f=\int_{\Theta}v(w)\pi_{F(w)}\chi_{w}^{\ast}\chi_{w}\pi_{F(w)}fd\mu(w).
	\end{equation*}
	By Cauchy-Schwartz inequality, we get 
	\begin{align*}
		\|f\|^{4}&=\langle f,f\rangle ^{2}=\langle \int_{\Theta}v(w)\pi_{F(w)}\chi_{w}^{\ast}\chi_{w}\pi_{F(w)}fd\mu(w),f\rangle^{2}\\&=\biggl(\int_{\Theta}v(w)\langle \chi_{w}\pi_{F(w)}f,\chi_{w}\pi_{F(w)}f\rangle d\mu(w)\biggr)^{2}\\&\leq \int_{\Theta}v^{2}(w)\|\chi_{w}\pi_{F(w)}f\|^{2}d\mu(w)\int_{\Theta}\|\chi_{w}\pi_{F(w)}f\|^{2}d\mu(w)\\&\leq \frac{1}{A}\|f\|^{2}\int_{\Theta}v^{2}(w)\|\chi_{w}\pi_{F(w)}f\|^{2}d\mu(w),
	\end{align*}
	then 
	\begin{equation*}
		A\|f\|^{2}\leq \int_{\Theta}v^{2}(w)\|\chi_{w}\pi_{F(w)}f\|^{2}d\mu(w),\quad\forall f\in U.
	\end{equation*}
	On the other hand 
	\begin{align*}
		\int_{\Theta}v^{2}(w)\|\chi_{w}\pi_{F(w)}f\|^{2}d\mu(w)&\leq B\int_{\Theta}\|\chi_{w}\pi_{F(w)}f\|^{2}d\mu(w)\\&\leq \frac{B}{A}\|f\|^{2},
	\end{align*}
	where $B=\sup_{w\in\Theta}v^{2}(w)$ and hence, $\{F(w),\chi_{w},v(w)\}_{w\in\Theta}$ is a continuous $g-$fusion frame.
\end{proof}
\section{Continuous $g-$atomic subspace}

\begin{definition}\label{def1}
	$\{F(w),\chi_{w},v(w)\}_{w\in\Theta}$ is said to be continuous $g-$atomic subspace of $U$ with respect to $K$ if 
	\begin{itemize}
		\item[(1)] $\{F(w),\chi_{w},v(w)\}_{w\in\Theta}$ is a continuous $g-$fusion bessel sequence in $U$.
		\item[(2)] Forall $f\in U$, there exists $\phi \in L^{2}(\Theta, U_{0})$ such that 
		\begin{equation*}
			Kf=\int_{\Theta}v(w)\pi_{F(w)}\chi_{w}^{\ast}\phi(w)d\mu(w),\quad and \quad \|\phi\|_{L^{2}(\Theta, U_{0})}\leq C\|f\|_{U},\quad for\quad some \quad C>0.
		\end{equation*}
	\end{itemize}
\end{definition}
\begin{theorem}\label{lem2}
	$\{F(w),\chi_{w},v(w)\}_{w\in\Theta}$ is a continuous $g-$atomic subspace of $U$ with respect to $K$ if and only if $\{F(w),\chi_{w},v(w)\}_{w\in\Theta}$ is a continuous $K-g-$fusion frame for $U$.
\end{theorem}
\begin{proof}
	Assum that $\{F(w),\chi_{w},v(w)\}_{w\in\Theta}$ is a continuous $g-$atomic subspace of $U$ with respect to $K$. Then $\{F(w),\chi_{w},v(w)\}_{w\in\Theta}$ is a continuous $g-$fusion bessel sequence in $U$, hence there exists $B>0$ such that 
	\begin{equation*}
		\int_{\Theta}v^{2}(w)\|\chi_{w}\pi_{F(w)}f\|^{2}d\mu(w)\leq B\|f\|^{2}, \quad\forall f\in U.
	\end{equation*}
	For any $f\in U$ we get
	\begin{equation*}
		\|K^{\ast}f\|=\sup_{\|g\|=1}\bigg|\langle K^{\ast}f,g\rangle \bigg|=\sup_{\|g\|=1}\bigg|\langle f,Kg\rangle \bigg|.
	\end{equation*}
	By definition \eqref{def1}, for $g\in U$ there exits $\phi\in L^{2}(\Theta,U_{0})$ such that 
	\begin{equation*}
		Kg=\int_{\Theta}v(w)\pi_{F(w)}\chi_{w}^{\ast}\phi(w)d\mu(w) \quad and \quad \|\phi\|_{L^{2}(\Theta, U_{0})}\leq C\|g\|_{U}\quad for \quad some\quad C>0.
	\end{equation*}
	Therefore 
	\begin{align*}
		\|K^{\ast}f\|&=\sup_{\|g\|=1}\bigg|\langle f,\int_{\Theta}v(w)\pi_{F(w)}\chi_{w}^{\ast}\phi(w)d\mu(w)\rangle \bigg|=\sup_{\|g\|=1}\bigg|\int_{\Theta}v(w)\langle \chi_{w}\pi_{F(w)}f,\phi(w)\rangle d\mu(w)\bigg|\\&\leq \sup_{\|g\|=1}\bigg(\int_{\Theta}v^{2}(w)\|\chi_{w}\pi_{F(w)}f\|^{2}d\mu(w)\bigg)^{\frac{1}{2}}\bigg(\int_{\Theta}\phi(w)\|^{2}d\mu(w)\bigg)^{\frac{1}{2}}\\&=\sup_{\|g\|=1}\|\phi(w)\|\bigg(\int_{\Theta}v^{2}(w)\|\chi_{w}\pi_{F(w)}f\|^{2}d\mu(w)\bigg)^{\frac{1}{2}}\\&\leq \sup_{\|g\|=1}\bigg(\int_{\Theta}v^{2}(w)\|\chi_{w}\pi_{F(w)}f\|^{2}d\mu(w)\bigg)^{\frac{1}{2}}C\|g\|\\&=C\bigg(\int_{\Theta}v^{2}(w)\|\chi_{w}\pi_{F(w)}f\|^{2}d\mu(w)\bigg)^{\frac{1}{2}},
	\end{align*}
	so
	\begin{equation*}
		\frac{1}{C^{2}}\|K^{\ast}f\|^{2}\leq \int_{\Theta}v^{2}(w)\|\chi_{w}\pi_{F(w)}f\|^{2}d\mu(w).
	\end{equation*}
	Conversely if $\{F(w),\chi_{w},v(w)\}_{w\in\Theta}$ is a continuous $K-g-$fusion frame for $U$ with bounds $A, B$ and $T_{\chi}$ is the corresponding synthesis operator.
	
	Then obviously $\{F(w),\chi_{w},v(w)\}_{w\in\Theta}$ is a continuous $g-$fusion bessel sequence in $U$. Now, for each $f\in U$ 
	\begin{equation*}
		A\|K^{\ast}f\|^{2}\leq \int_{\Theta}v^{2}(w)\|\chi_{w}\pi_{F(w)}f\|^{2}d\mu(w)=\|T_{\chi}^{\ast}f\|^{2},
	\end{equation*}
	implies $AKK^{\ast}\leq T_{\chi}T_{\chi}^{\ast}$. By theorem \ref{lem1}, exists $L\in B(U,L^{2}(\Theta,U_{0}))$ such that $K=T_{\chi}L$. Define $L(f)=\phi$, for all $f\in U$. Then for each $f\in U$ we get
	\begin{equation*}
		K(f)=T_{\chi}L(f)=T_{\chi}(\phi)=\int_{\Theta}v(w)\pi_{F(w)}\chi_{w}^{\ast}\phi(w)d\mu(w)	
	\end{equation*}
	and 
	\begin{equation*}
		\|\phi\|_{L^{2}(\Theta,U_{0})}=\|L(f)\|_{L^{2}(\Theta,U_{0})}\leq C\|f\|_{U}, \quad for \quad some \quad C>0.
	\end{equation*}
	Therefore, $\{F(w),\chi_{w},v(w)\}_{w\in\Theta}$ is a continuous $g-$atomic subspace of $U$ with respect to $K$.
\end{proof}
\begin{theorem}
	Let $\{F(w),\chi_{w},v(w)\}_{w\in\Theta}$ be a continuous $g-$fusion frame for $U$. Then $\{F(w),\chi_{w},v(w)\}_{w\in\Theta}$ is a continuous $g-$atomic subspace of $U$ with respect to its continuous $g-$fusion frame operator $S_{\chi}$.
\end{theorem}	
\begin{proof}
	Suppose that  $\{F(w),\chi_{w},v(w)\}_{w\in\Theta}$ is a continuous $g-$fusion frame in $U$, then there exist $A, B>0$ such that 
	\begin{equation*}
		A\|f\|^{2}\leq \int_{\Theta}v^{2}(w)\|\chi_{w}\pi_{F(w)}f\|^{2}d\mu(w)\leq B\|f\|^{2},\quad \forall f\in U.
	\end{equation*}
	We have $\mathcal{R}(T_{\chi})=U=\mathcal{R}(S_{\chi})$ by theorem \ref{lem1}, there exists $\alpha>0$ such that $\alpha S_{\chi}S_{\chi}^{\ast}\leq T_{\chi}T_{\chi}^{\ast}$ hence for each $f\in U$ we get 
	\begin{equation*}
		\alpha\|S_{\chi}^{\ast}f\|^{2}\leq \|T_{\chi}^{\ast}f\|^{2}=\int_{\Theta}v^{2}(w)\|\chi_{w}\pi_{F(w)}f\|^{2}d\mu(w)\leq B\|f\|^{2}.
	\end{equation*}
\end{proof}
\begin{theorem}
	Let $\{F(w),\chi_{w},v(w)\}_{w\in\Theta}$ and $\{F(w),\xi_{w},v(w)\}_{w\in\Theta}$ be two continuous $g-$atomic subspaces of $U$ with respect to $K\in B(U)$ with the corresponding synthesis operators $T_{\chi}$ and $T_{\xi}$, respectively. If $T_{\chi}T_{\xi}^{\ast}=\theta_{H}$($\theta_{U}$ is a null operator on $U$) and $L,G\in B(U)$ such that $L+G$ is invertible operator on $U$ with $K(L+G)=(L+G)K$, then $\{(L+G)F(w),(\chi_{w}+\xi_{w})\pi_{F(w)}(L+G)^{\ast},v(w)\}_{w\in\Theta}$ is a continuous $g-$atomic subspace of $U$ with respect to $K$.
\end{theorem}
\begin{proof}
	We have $\{F(w),\chi_{w},v(w)\}_{w\in\Theta}$ and $\{F(w),\xi_{w},v(w)\}_{w\in\Theta}$ are continuous $g-$atomic subspaces with respect to $K$ by theorem \ref{lem2} they are continuous $K-g-$fusion frame for $U$, hence for each $f\in U$ there exist positive constants $(A_{1},B_{1})$ and $(A_{2},B_{2})$ such that 
	\begin{equation*}
		A_{1}\|K^{\ast}f\|^{2}\leq \int_{\Theta}v^{2}(w)\|\chi_{w}\pi_{F(w)}f\|^{2}d\mu(w)\leq B_{1}\|f\|^{2},
	\end{equation*}
	and 
	\begin{equation*}
		A_{2}\|K^{\ast}f\|^{2}\leq \int_{\Theta}v^{2}(w)\|\xi_{w}\pi_{F(w)}f\|^{2}d\mu(w)\leq B_{2}\|f\|^{2}.
	\end{equation*}
	We have $T_{\chi}T_{\xi}^{\ast}=\theta_{U}$, for any $f\in U$ we get 
	
	\begin{equation*}
		T_{\chi}\{v(w)\xi_{w}\pi_{F(w)}f\}_{w\in\Theta}=\int_{\Theta}v^{2}(w)\pi_{F(w)}\chi_{w}^{\ast}\xi_{w}\pi_{F(w)}fd\mu(w).
	\end{equation*}
	Since $L+G$ invertible, then 
	\begin{equation*}
		\|K^{\ast}f\|^{2}=\|((L+G)^{-1})^{\ast}(L+G)^{\ast}K^{\ast}f\|^{2}\leq \|(L+G)^{-1}\|^{2}\|(L+G)^{\ast}K^{\ast}f\|^{2},
	\end{equation*}
	for each $f\in U$ we get
	\begin{align*}
		\int_{\Theta}v^{2}(w)&\|(\chi_{w}+\xi_{w})\pi_{F(w)}(L+G)^{\ast}\pi_{(L+G)F(w)}f\|^{2}d\mu(w)\\&=\int_{\Theta}v^{2}(w)\|(\chi_{w}+\xi_{w})\pi_{F(w)}(L+G)^{\ast}f\|^{2}d\mu(w)\\&=\int_{\Theta}v^{2}(w)\langle (\chi_{w}+\xi_{w})\pi_{F(w)}T^{\ast}f,(\chi_{w}+\xi_{w})\pi_{F(w)}T^{\ast}f\rangle d\mu(w), \quad (T=L+G)\\&=\int_{\Theta}v^{2}(w)\bigg[\|\chi_{w}\pi_{F(w)}T^{\ast}f\|^{2}+\|\xi_{w}\pi_{F(w)}T^{\ast}f\|^{2}+2\mathcal{R}e\langle T\pi_{F(w)}\chi_{w}^{\ast}\xi_{w}\pi_{F(w)}T^{\ast}f,f\rangle \bigg]d\mu(w)\\&=\int_{\Theta}v^{2}(w)\|\chi_{w}\pi_{F(w)}T^{\ast}f\|^{2}d\mu(w)+\int_{\Theta}v^{2}(w)\|\xi_{w}\pi_{F(w)}T^{\ast}f\|^{2}d\mu(w)\\&\leq B_{1}\|T^{\ast}f\|^{2}+B_{2}\|T^{\ast}f\|^{2}\\&\leq (B_{1}+B_{2})\|L+G\|^{2}\|f\|^{2}.
	\end{align*}
	On the other hand, 
	\begin{align*}
		\int_{\Theta}v^{2}(w)\|(\chi_{w}+\xi_{w})\pi_{F(w)}(L+G)^{\ast}\pi_{(L+G)F(w)}f\|^{2}d\mu(w)&\geq \int_{\Theta}v^{2}(w)\|\chi_{w}\pi_{F(w)}(L+G)^{\ast}f\|^{2}d\mu(w)\\&\geq A_{1}\|K^{\ast}(L+G)^{\ast}f\|^{2}\\&=A_{1}\|(L+G)^{\ast}K^{\ast}f\|^{2}\\&\geq A_{1}\|(L+G)^{-1}\|^{-2}\|K^{\ast}f\|^{2}.
	\end{align*}
	So $\{(L+G)F(w),(\chi_{w}+\xi_{w})\pi_{F(w)}(L+G)^{\ast},v(w)\}_{w\in\Theta}$ is a continuous $K-g-$fusion frame and by theorem \ref{lem2}, it is a continuous $g-$atomic subspace of $U$ with respect to $K$.
\end{proof}
\begin{theorem}
	Let $\{F(w),\chi_{w},v(w)\}_{w\in\Theta}$ is a continuous $g-$atomic subspace for $K\in B(U)$ and $S_{\chi}$ be the frame operator of $\chi$. If $L\in B(U)$ is a positive and invertible operator on $U$, then $\{(I_{U}+L)F(w),\chi_{w}\pi_{F(w)}(I_{U}+L)^{\ast},v(w)\}_{w\in\Theta}$ is a continuous $g-$atomic subspace of $U$ with respect to $K$.
\end{theorem}
Because $\{F(w),\chi_{w},v(w)\}_{w\in\Theta}$ is a continuous $g-$atomic subspace with respect to $K$, by theorem \ref{lem2} it is a continuous $K-g-$fusion frame for $U$ then by theorem \ref{lem3}, there exists $A>0$ such that $S_{\Lambda}\geq AKK^{\ast}$. For each $f\in U$ we get 
\begin{align*}
	\int_{\Theta}v^{2}(w)\|\chi_{w}\pi_{F(w)}(I_{U}+L)^{\ast}\pi_{(L+G)F(w)}f\|^{2}d\mu(w)&=\int_{\Theta}v^{2}(w)\|\chi_{w}\pi_{F(w)}(I_{U}+L)^{\ast}f\|^{2}d\mu(w)\\&\leq B\|(I_{U}+L)^{\ast}f\|^{2}\\&\leq B\|I_{U}+L\|^{2}\|f\|^{2}
\end{align*}
So, $\{(I_{U}+L)F(w),\chi_{w}\pi_{F(w)}(I_{U}+L)^{\ast},v(w)\}_{w\in\Theta}$ is a continuous $g-$fusion bessel sequence in $U$. Also, for all $f\in U$ we get 
\begin{align*}
	\int_{\Theta}&v^{2}(w)\pi_{(I_{U}+L)F(w)}(\chi_{w}\pi_{F(w)}(I_{U}+L)^{\ast})^{\ast}\chi_{w}\pi_{F(w)}(I_{U}+L)^{\ast}\pi_{(I_{U}+L)F(w)}fd\mu(w)\\&=\int_{\Theta}v^{2}(w)\pi_{(I_{U}+L)F(w)}(I_{U}+L)\pi_{F(w)}\chi_{w}^{\ast}\chi_{w}\pi_{F(w)}(I_{U}+L)^{\ast}\pi_{(I_{U}+L)F(w)}fd\mu(w)\\&=\int_{\Theta}v^{2}(w)(\pi_{F(w)}(I_{U}+L)^{\ast}\pi_{(I_{U}+L)F(w)})^{\ast}\chi_{w}^{\ast}\chi_{w}\pi_{F(w)}(I_{U}+L)^{\ast}\pi_{(I_{U}+L)F(w)}fd\mu(w)\\&=\int_{\Theta}v^{2}(w)(\pi_{F(w)}(I_{U}+L)^{\ast})^{\ast}\chi_{w}^{\ast}\chi_{w}\pi_{F(w)}(I_{U}+L)^{\ast}fd\mu(w)\\&=\int_{\Theta}v^{2}(w)(I_{U}+L)\pi_{F(w)}\chi_{w}^{\ast}\chi_{w}\pi_{F(w)}(I_{U}+L)^{\ast}fd\mu(w)\\&=(I_{U}+L)S_{\chi}(I_{U}+L)^{\ast}f.
\end{align*}
This implies the frame operator of $\{(I_{U}+L)F(w),\chi_{w}\pi_{F(w)}(I_{U}+L)^{\ast},v(w)\}_{w\in\Theta}$ is $(I_{U}+L)S_{\chi}(I_{U}+L)^{\ast}$.

\section{Frame operator for a pair of continuous $g-$fusion bessel sequences}

\begin{definition}
	Let $\chi=\{F(w),\chi_{w},v(w)\}_{w\in\Theta}$ and $\xi=\{G(w),\xi_{w},s(w)\}_{w\in\Theta}$ be two continuous $g-$bessel sequences in $U$ with bounds $D_{1}$ and $D_{2}$. Then the operator $S_{\xi\chi}:U\rightarrow U$ defined by 
	\begin{equation*}
		S_{\xi\chi}f=\int_{\Theta}v(w)s(w)\pi_{G(w)}\xi_{w}^{\ast}\chi_{w}\pi_{F(w)}fd\mu(w).
	\end{equation*}
\end{definition}
is called the frame operator for the pair of continuous $g-$bessel sequences $\chi$ and $\xi$.
\begin{theorem}
	The frame operator $S_{\xi\chi}$ for the pair of continuous $g-$fusion bessel sequences $\chi$ and $\xi$ is bounded and $S_{\xi\chi}^{\ast}=S_{\chi\xi}$.
\end{theorem}
\begin{proof}
	For each $f,g\in U$ we have 
	\begin{align*}
		\langle S_{\xi\chi}f,g\rangle&=\langle \int_{\Theta}v(w)s(w)\pi_{G(w)}\xi_{w}^{\ast}\chi_{w}\pi_{F(w)}fd\mu(w),g\rangle \\&=\int_{\Theta}v(w)s(w)\langle \chi_{w}\pi_{F(w)}f,\xi_{w}\pi_{G(w)}g\rangle d\mu(w).
	\end{align*}
	By the Cauchy-Schwartz inequality, we have
	\begin{align*}
		\bigg|\langle S_{\xi\chi}f,g\rangle \bigg|&\leq \bigg(\int_{\Theta}s^{2}(w)\|\xi_{w}\pi_{G(w)}g\|^{2}d\mu(w)\bigg)^{\frac{1}{2}}\bigg(\int_{\Theta}v^{2}(w)\|\chi_{w}\pi_{F(w)}f\|^{2}d\mu(w)\bigg)^{\frac{1}{2}}\\&\leq \sqrt{D_{1}}\|g\|\sqrt{D_{2}}\|f\|,
	\end{align*}
	this implies that $S_{\xi\chi}$ is a bounded operator with $\|S_{\xi\chi}\|\leq \sqrt{D_{1}D_{2}}$. Now 
	\begin{align*}
		\|S_{\xi\chi}f\|&=\sup_{\|g\|=1}\bigg|\langle S_{\xi\chi}f,g\rangle \bigg|\\&\leq \sup_{\|g\|=1}\sqrt{D_{2}}\|g\|\bigg(\int_{\Theta}v^{2}(w)\|\chi_{w}\pi_{F(w)}f\|^{2}d\mu(w)\bigg)^{\frac{1}{2}}\\&\leq \sqrt{D_{2}}\bigg(\int_{\Theta}v^{2}(w)\|\chi_{w}\pi_{F(w)}f\|^{2}d\mu(w)\bigg)^{\frac{1}{2}}
	\end{align*}
	and similarly, it can be schown that 
	\begin{equation*}
		\|S_{\xi\chi}^{\ast}g\|\leq \sqrt{D_{1}}\bigg(\int_{\Theta}v^{2}(w)\|\xi_{w}\pi_{G(w)}g\|^{2}d\mu(w)\bigg)^{\frac{1}{2}}.
	\end{equation*}
\end{proof}
Moreover, for each $f,g\in U$ we have 
\begin{align*}
	\langle S_{\xi\chi}f,g\rangle&=\langle \int_{\Theta}v(w)s(w)\pi_{G(w)}\xi_{w}^{\ast}\chi_{w}\pi_{F(w)}fd\mu(w),g\rangle\\&=\int_{\Theta}v(w)s(w)\langle f,\pi_{F(w)}\chi_{w}^{\ast}\xi_{w}\pi_{G(w)}g\rangle d\mu(w)\\&=\langle f,\int_{\Theta}v(w)s(w)\pi_{F(w)}\chi_{w}^{\ast}\xi_{w}\pi_{G(w)}gd\mu(w)\rangle \\&=\langle f,S_{\chi\xi}g\rangle.
\end{align*}
So, $S_{\xi\chi}^{\ast}=S_{\chi\xi}$.
\begin{theorem}
	Let $S_{\xi\chi}$ be the frame operator for a pair of continuous $g-$fusion bessel sequences $\chi$ and $\xi$ with bounds $D_{1}$ and $D_{2}$, respectively. Then the following statements are equivalent:
	\begin{itemize}
		\item[(1)] is bounded below.
		\item[(2)] There exists $K\in B(U)$ such that $\{T_{w}\}_{w\in\Theta}$ is a continuous resolution of the identity operator on $U$, where $T_{w}=v(w)s(w)K\pi_{G(w)}\xi_{w}^{\ast}\chi_{w}\pi_{F(w)}$, $w\in \Theta$.
	\end{itemize}
	If one of the given conditions holds, then $\chi$ is a continuous $g-$fusion frame.
\end{theorem}
\begin{proof}
	$(1)\implies (2)$	Assum that $S_{\xi\chi}$ is bounded below. Then for each $f\in U$ there exists $A>0$ such that 
	\begin{equation*}
		\|f\|^{2}\leq A\|S_{\xi\chi}f\|^{2},
	\end{equation*}
	so
	\begin{equation*}
		\langle I_{U}f,f\rangle \leq A\langle S_{\chi\xi}^{\ast}S_{\chi\xi}f,f\rangle, 
	\end{equation*}
	hence 
	\begin{equation*}
		I_{U}^{\ast}I_{U}\leq AS_{\xi\chi}^{\ast}S_{\xi\chi}.
	\end{equation*}
	By theorem \ref{lem1}, there exists $K\in B(U)$ such that $KS_{\xi\chi}=I_{U}$. 
	
	Then for each $f\in U$ we get
	\begin{align*}
		f&=KS_{\xi\chi}f\\&=K\int_{\Theta}v^{2}(w)s^{2}(w)\pi_{G(w)}\xi_{w}^{\ast}\chi_{w}\pi_{F(w)}fd\mu(w)\\&=\int_{\Theta}v^{2}(w)s^{2}(w)K\pi_{G(w)}\xi_{w}^{\ast}\chi_{w}\pi_{F(w)}fd\mu(w)\\&=\int_{\Theta}T_{w}fd\mu(w).
	\end{align*}
	$(2)\implies (1)$ Because $\{T_{w}\}_{w\in\Theta }$ is a continuous resolution of the identity operator on $U$ for any $f\in U$ we get 
	\begin{align*}
		f&=\int_{\Theta}T_{w}(f)d\mu(w)\\&=\int_{\Theta}v(w)s(w)K\pi_{G(w)}\xi^{\ast}\chi_{w}\pi_{F(w)}fd\mu(w)\\&=K\int_{\Theta}v(w)s(w)\pi_{G(w)}\xi_{w}^{\ast}\chi_{w}\pi_{F(w)}fd\mu(w)\\&=KS_{\xi\chi}f.
	\end{align*}
	Then, $I_{U}=KS_{\xi\chi}$. Hence, by theorem \ref{lem1}, there exists $\alpha>0$ such that $I_{U}I_{U}^{\ast}\leq \alpha S_{\xi\chi}S_{\xi\chi}^{\ast}$ so $S_{\xi\chi}$ is bounded below.
	
	Last part: we assum that $S_{\xi\chi}$ is bounded below. So for all $f\in U$ there exists $M>0$ such that $\|S_{\xi\chi}f\|\geq M\|f\|$ then
	\begin{equation*}
		M^{2}\|f\|^{2}\leq \|S_{\xi\chi}f\|^{2}\leq D_{2}\int_{\Theta}v^{2}(w)\|\chi_{w}\pi_{F(w)}f\|^{2}d\mu(w),
	\end{equation*}
	so
	\begin{equation*}
		\frac{M^{2}}{D_{2}}\|f\|^{2}\leq \int_{\Theta}v^{2}(w)\|\chi_{w}\pi_{F(w)}f\|^{2}d\mu(w).
	\end{equation*}
	Hence, $\chi$ is a continuous $g-$fusion frame for $U$ with bounds $\frac{M^{2}}{D_{2}}$ and $D_{1}$.
\end{proof}	
\begin{theorem}
	Let $S_{\xi\chi}$ be the frame operator for a pair of continuous $g-$fusion bessel sequences $\chi$ and $\xi$ with bounds $D_{1}$ and $D_{2}$, respectively. Suppose $\lambda_{1}<1, \lambda_{2}>-1$ such that for each $f\in U$, $\|f-S_{\xi\chi}f\|\leq \lambda_{1}\|f\|+\lambda_{2}\|S_{\xi\chi}f\|$. Then $\chi$ is a continuous $g-$fusion frame for $U$.
\end{theorem}	
\begin{proof}
	For all $f\in U$ we get 
	\begin{equation*}
		\|f\|-\|S_{\xi\chi}f\|\leq \|f-S_{\xi\chi}f\|\leq \lambda_{1}\|f\|+\lambda_{2}\|S_{\xi\chi}f\|,
	\end{equation*}
	then
	\begin{equation*}
		(1-\lambda_{1})\|f\|\leq (1+\lambda_{2})\|S_{\xi\chi}f\|,
	\end{equation*}
	so
	\begin{equation*}
		\bigg(\frac{1-\lambda_{1}}{1+\lambda_{2}}\bigg)\|f\|\leq \sqrt{D_{2}}\bigg(\int_{\Theta}v^{2}(w)\|\chi_{w}\pi_{F(w)}(f)\|^{2}d\mu(w)\bigg)^{\frac{1}{2}}
	\end{equation*}
	therefore
	\begin{equation}\label{eq2}
		\frac{1}{D_{2}}\bigg(\frac{1-\lambda_{1}}{1+\lambda_{2}}\bigg)^{2}\|f\|^{2}\leq \int_{\Theta}v^{2}(w)\|\chi_{w}\pi_{F(w)}f\|^{2}d\mu(w).
	\end{equation}
	Hence, $\chi$ is a continuous $g-$fusion frame for $U$ with bounds $(1-\lambda_{1})^{2}(1+\lambda_{2})^{-2}D_{2}^{-1}$ and $D_{1}$.
\end{proof}	
\begin{theorem}
	Let $S_{\xi\chi}$ be the frame operator for a pair of continuous $g-$fusion bessel sequences  $\Lambda$ and $\Gamma$ of the bounds $D_{1}$ and $D_{2}$, respectively. Assume $\lambda\in [0,1)$ such that 
	\begin{equation*}
		\|f-S_{\xi\chi}f\|\leq \lambda\|f\|, \quad\forall f\in U.	
	\end{equation*}
	Then $\chi$ and $\xi$ are continuous $g-$fusion frame for $U$.
\end{theorem}
\begin{proof}
	By putting $\lambda_{1}=\lambda$ and $\lambda_{2}=0$ in \eqref{eq2}, we have 
	\begin{equation*}
		\frac{(1-\lambda)^{2}}{D_{2}}\|f\|^{2}\leq \int_{\Omega}v^{2}(w)\|\chi_{w}\pi_{F(w)}(f)\|^{2}d\mu(w),
	\end{equation*}
	and hence $\chi$ is a continuous $g-$fusion frame. Now, for all $f\in U$ we get 
	\begin{equation*}
		\|f-S_{\xi\chi}^{\ast}f\|=\|(I_{U}-S_{\xi\chi})^{\ast}\|\leq \|(I_{U}-S_{\xi\chi})\|\|f\|\leq \lambda\|f\|,
	\end{equation*}
	then
	\begin{equation*}
		(1-\lambda)\|f\|\leq \|S_{\xi\chi}^{\ast}f\|\leq \sqrt{D_{1}}\bigg(\int_{\Theta}s^{2}(w)\|\xi_{w}\pi_{G(w)}f\|^{2}d\mu(w)\bigg)^{\frac{1}{2}}
	\end{equation*}
	so
	\begin{equation*}
		\int_{\theta}s^{2}(w)\|\xi_{w}\pi_{G(w)}f\|^{2}d\mu(w)\geq \frac{(1-\lambda)^{2}}{D_{1}}\|f\|^{2}, \quad \forall f\in U.
	\end{equation*}
	Therefore, $\xi$ is a continuous $g-$fusion frame with bounds $\frac{(1-\lambda)^{2}}{D_{1}}$ and $D_{2}$.
\end{proof}	

\begin{definition}
	Let $U$ and $X$ be two Hilbert spaces. Define 
	\begin{equation*}
		U\oplus X=\{(f,g):f\in U,g\in X\}
	\end{equation*}
	Then $U\oplus X$ froms a Hilbert space with respect to point-wise operations and inner product defined by 
	\begin{equation*}
		\langle (f,g),(f^{'},g^{'})\rangle =\langle f,f^{'}\rangle_{U}+\langle g,g^{'}\rangle_{X},\quad \forall f,f^{'}\in U \quad and \quad \forall g,g^{'}\in X.
	\end{equation*}
	Now, if $L\in B(U,Z), T\in B(X,Y)$, then for all $f\in U$, $g\in X$ we define 
	\begin{equation*}
		L\oplus T\in B(U\oplus X, Z\oplus Y)\quad by \quad (L\oplus T)(f,g)=(Lf,Tg),
	\end{equation*}
	and $(L\oplus T)^{\ast}=L^{\ast}\oplus T^{\ast}$, where $Z, Y$ are Hilbert spaces and also we define 
	$\pi_{F(w)\oplus G(w)}(f,g)=(\pi_{F(w)}f,\pi_{G(w)}f)$, where $\pi_{F(w)}, \pi_{G(w)}$ and $\pi_{F(w)\oplus G(w)}$ are orthonormal projections onto the closed subspaces $F(w)\subset U$, $G(w)\subset X$ and $F(w)\oplus G(w)\subset U\oplus X$, respectively.
	
	From here we assume that for each $w\in\Theta$, $F(w)\oplus G(w)$ are the closed subspaces of $U\oplus X$ and $\Gamma_{w}\in B(X,X_{w})$, where $\{X_{w}\}_{w\in\Theta}$ is the collection of the spaces and $\chi_{w}\oplus \xi_{w}\in B(U\oplus X,U_{w}\oplus X_{w})$.
\end{definition}	

\begin{theorem}
	Let $\chi=\{F(w),\chi_{w},v(w)\}_{w\in\Theta}$ be a continuous $g-$fusion frame for $U$ with bounds $A,B$ and $\xi=\{G(w), \xi_{w},v(w)\}_{w\in\Theta}$ be a continuous $g-$fusion frame for $U$ with bounds for $C,D$. Then $\chi\oplus \xi=\{F(w)\oplus G(w),\chi_{w}\oplus \xi_{w},v(w)\}_{w\in\Theta}$ is a continuous $g-$fusion frame for $U\oplus X$ with bounds $\min\{A,C\}$, $\max\{B,D\}$. Furthermore, if $S_{\chi}, S_{\xi}$ and $S_{\chi\oplus \xi}$ are continuous $g-$fusion frame operators for $\chi,\xi,\chi\oplus \xi$, respectively, then we have $S_{\chi\oplus \xi}=S_{\chi}\oplus S_{\xi}$.
\end{theorem}
\begin{proof}
	Let $(f,g)\in U\oplus X$. Then 
	\begin{align*}
		\int_{\Theta}v^{2}(w)&\|(\chi_{w}\oplus \xi_{w})\pi_{F(w)\oplus G(w)}(f,g)\|^{2}d\mu(w)\\&=\int_{\theta}v^{2}(w)\langle (\chi_{w}\oplus \xi_{w})\pi_{F(w)\oplus G(w)}(f,g),(\chi_{w}\oplus \xi_{w})\pi_{F(w)\oplus G(w)}(f,g)\rangle d\mu(w)\\&=\int_{\Theta}v^{2}(w)\langle (\chi_{w}\oplus \xi_{w})(\pi_{F(w)}f,\pi_{G(w)}g),(\chi_{w}\oplus \xi_{w})(\pi_{F(w)}f,\pi_{G(w)}g)\rangle d\mu(w)\\&=\int_{\theta}v^{2}(w)\langle (\chi_{w}\pi_{F(w)}f,\xi_{w}\pi_{G(w)}g),(\chi_{w}\pi_{F(w)}f,\xi_{w}\pi_{G(w)}g)\rangle d\mu(w)\\&=\int_{\theta}v^{2}(w)\bigg(\langle \chi_{w}\pi_{F(w)}f,\chi_{w}\pi_{F(w)}f\rangle_{U}+\langle \xi_{w}\pi_{G(w)}g,\xi_{w}\pi_{G(w)}g\rangle_{X}\bigg)d\mu(w)\\&=\int_{\theta}v^{2}(w)\|\chi_{w}\pi_{F(w)}f\|_{U}^{2}d\mu(w)+\int_{\theta}v^{2}(w)\|\xi_{w}\pi_{G(w)}g\|_{X}^{2}d\mu(w)\\&\leq B\|f\|_{U}^{2}+D\|g\|_{X}^{2}\\&\leq \max\{B,D\}\|(f,g)\|_{U\oplus X}^{2}.
	\end{align*}
	Similarly, it can be shown that 
	\begin{equation*}
		\min\{A,C\}\|(f,g)\|_{U\oplus X}^{2}\leq \int_{\Theta}v^{2}(w)\|(\chi_{w}\oplus \xi_{w})\pi_{F(w)\oplus G(w)}(f,g)\|^{2}d\mu(w).
	\end{equation*}
	Moreover, for $(f,g)\in U\oplus X$, we get 
	\begin{align*}
		S_{\chi\oplus \xi}(f,g)&=\int_{\Theta}v^{2}(w)\pi_{F(w)\oplus G(w)}(\chi_{w}\oplus \xi_{w})^{\ast}(\chi_{w}\oplus \xi_{w})\pi_{F(w)\oplus G(w)}(f,g)d\mu(w)\\&=\int_{\theta}v^{2}(w)\pi_{F(w)\oplus G(w)}(\chi_{w}\oplus \xi_{w})^{\ast}(\chi_{w}\oplus \xi_{w})(\pi_{F(w)}f,\pi_{G(w)}g)d\mu(w)\\&=\int_{\Theta}v^{2}(w)(\chi_{w}\oplus \xi_{w})^{\ast}(\chi_{w}\pi_{F(w)}f,\xi_{w}\pi_{G(w)}g)d\mu(w)\\&=\int_{\Theta}v^{2}(w)\pi_{F(w)\oplus G(w)}(\chi_{w}^{\ast}\oplus \xi_{w}^{\ast})(\chi_{w}\pi_{F(w)}f,\xi_{w}\pi_{G(w)}g)d\mu(w)\\&=\int_{\Theta}v^{2}(w)\pi_{F(w)\oplus G(w)}\bigg(\chi_{w}^{\ast}\chi_{w}\pi_{F(w)}f,\xi_{w}^{\ast}\xi_{w}\pi_{G(w)}g\bigg)d\mu(w)\\&=\int_{\theta}v^{2}(w)\bigg(\pi_{F(w)}\chi_{w}^{\ast}\chi_{w}\pi_{F(w)}f,\pi_{G(w)}\xi_{w}^{\ast}\Gamma_{w}\pi_{G(w)}g\bigg)d\mu(w)\\&=\bigg(\int_{\Theta}v^{2}(w)\pi_{F(w)}\chi_{w}^{\ast}\chi_{w}\pi_{F(w)}fd\mu(w),\int_{\Theta}v^{2}(w)\pi_{G(w)}\xi_{w}^{\ast}\xi_{w}\pi_{G(w)}gd\mu(w)\bigg)\\&=(S_{\chi}f,S_{\xi}g)\\&=(S_{\chi}\oplus S_{\xi})(f,g).
	\end{align*}
	Then, $S_{\chi\oplus \xi}=S_{\chi}\oplus S_{\xi}$.
\end{proof}
\begin{theorem}
	Let $\chi\oplus \xi=\{F(w)\oplus G(w),\chi_{w}\oplus \xi_{w},v(w)\}_{w\in\Theta}$ be a continuous $g-$fusion frame for $U\oplus X$ with frame operator $S_{\chi\oplus \xi}$. Then 
	\begin{equation*}
		\delta=\{S_{\chi\oplus \xi}^{-\frac{1}{2}}(F(w)\oplus G(w)),(\chi_{w}\oplus \xi_{w})\pi_{F(w)\oplus G(w)}S_{\chi\oplus \xi}^{-\frac{1}{2}},v(w)\}_{w\in \Theta}
	\end{equation*}
	is a continuous Parseval $g-$fusion frame for $U\oplus X$.
\end{theorem}

\begin{proof}
	Because $S_{\chi\oplus \xi}$ is a positive operator, there exists a unique positive square root $S_{\chi\oplus \xi}^{\frac{1}{2}}$ (or $S_{\chi\oplus \xi}^{-\frac{1}{2}}$) and they commute with $S_{\chi\oplus \xi}$ and $S_{\chi\oplus \xi}^{-1}$. Hence, all  $(f,g)\in U\oplus X$ we get 
	\begin{align*}
		(f,g)&=S_{\chi\oplus \xi}^{-\frac{1}{2}}S_{\chi\oplus \xi}S_{\chi\oplus \xi}^{-\frac{1}{2}}(f,g)\\&=\int_{\Theta}v^{2}(w)S_{\chi\oplus \xi}^{-\frac{1}{2}}\pi_{F(w)\oplus G(w)}(\chi_{w}\oplus \xi_{w})^{\ast}(\chi_{w}\oplus \xi_{w})\pi_{F(w)\oplus G(w)}S_{\chi\oplus \xi}^{-\frac{1}{2}}(f,g)d\mu(w).
	\end{align*}
	For all $(f,g)\in U\oplus X$ we get 
	\begin{align*}
		\|(f,g)\|^{2}&=\langle (f,g),(f,g)\rangle\\&=\langle \int_{\Theta}v^{2}(w)S_{\chi\oplus \xi}^{-\frac{1}{2}}\pi_{F(w)\oplus G(w)}(\chi_{w}\oplus \xi_{w})^{\ast}(\chi_{w}\oplus \xi_{w})\pi_{F(w)\oplus G(w)}S_{\chi\oplus\xi}^{-\frac{1}{2}}(f,g)d\mu(w), (f,g)\rangle \\&=\int_{\Theta}v^{2}(w)\langle (\chi_{w}\oplus \xi_{w})\pi_{F(w)\oplus G(w)}S_{\chi\oplus \xi}^{-\frac{1}{2}}(f,g),(\chi_{w}\oplus \xi_{w})\pi_{F(w)\oplus G(w)}S_{\chi\oplus \xi}^{-\frac{1}{2}}(f,g)\rangle d\mu(w)\\&=\int_{\Theta}v^{2}(w)\|(\chi_{w}\oplus \xi_{w})\pi_{F(w)\oplus G(w)}S_{\chi\oplus \xi}^{-\frac{1}{2}}(f,g)\|^{2}d\mu(w)\\&=\int_{\theta}v^{2}(w)\|(\chi_{w}\oplus \xi_{w})\pi_{F(w)\oplus G(w)}S_{\chi\oplus\xi}^{-\frac{1}{2}}\pi_{S_{\chi\oplus\xi}^{-\frac{1}{2}}(F(w)\oplus G(w))}(f,g)\|^{2}d\mu(w). 
	\end{align*}
	This implies that $\delta$ is a continuous Parseval $g-$fusion frame for $U\oplus X$.
\end{proof}
\begin{theorem}
	Let $\chi\oplus \xi=\{F(w)\oplus G(w),\chi_{w}\oplus \xi_{w},v(w)\}_{w\in\Theta}$ be a continuous $g-$fusion frame for $U\oplus X$ with bounds $A_{1}, B_{1}$ and $S_{\chi\oplus \xi}$ be the corresponding frame operator. Then 
	\begin{equation*}
		\delta=\{S_{\chi\oplus \xi}^{-1}(F(w)\oplus G(w)),(\chi_{w}\oplus \xi_{w})\pi_{F(w)\oplus G(w)}S_{\chi\oplus \xi}^{-1},v(w)\}_{w\in \theta}
	\end{equation*}
	is a continuous $g-$fusion frame for $U\oplus X$ with frame operator $S_{\chi\oplus \xi}^{-1}$.
\end{theorem}
\begin{proof}
	For all $(f,g)\in U\oplus X$ we get 
	\begin{align*}
		(f,g)&=S_{\chi\oplus \xi}S_{\chi\oplus \xi}^{-1}(f,g)\\&=\int_{\Theta}v^{2}(w)\pi_{F(w)\oplus G(w)}(\chi_{w}\oplus\xi_{w})^{\ast}(\chi_{w}\oplus \xi_{w})\pi_{F(w)\oplus G(w)}S_{\chi\oplus \xi}^{-1}(f,g)d\mu(w)
	\end{align*}
	By theorem \ref{lem4}, for each $(f,g)\in U\oplus X$ we get 
	\begin{align*}
		\int_{\theta}v^{2}(w)&\|(\chi_{w}\oplus \xi_{w})\pi_{F(w)\oplus G(w)}S_{\chi\oplus \xi}^{-1}\pi_{S_{\chi\oplus\xi}^{-1}(F(w)\oplus G(w))}(f,g)\|^{2}d\mu(w)\\&=\int_{\Theta}v^{2}(w)\|(\chi_{w}\oplus\xi_{w})\pi_{F(w)\oplus G(w)}S_{\chi\oplus \xi}^{-1}(f,g)\|^{2}d\mu(w)\\&\leq B_{1}\|S_{\chi\oplus \xi}^{-1}\|^{2}\|(f,g)\|^{2}.
	\end{align*}
	On the other hand, we have 
	\begin{align*}
		\|(f,g)\|^{4}&=\bigg|\langle (f,g),(f,g)\rangle \bigg|^{2}\\&=\bigg|\langle \int_{\Theta}v^{2}\pi_{F(w)\oplus G(w)}(\chi_{w}\oplus \xi_{w})^{\ast}(\chi_{w}\oplus \xi_{w})\pi_{F(w)\oplus G(w)}S_{\chi\oplus\xi}^{-1}(f,g)d\mu(w),(f,g)\rangle \bigg|^{2}\\&=\bigg|\int_{\Theta}v^{2}(w)\langle (\chi_{w}\oplus \xi_{w})\pi_{F(w)\oplus G(w)}S_{\chi\oplus \xi}^{-1}(f,g),(\chi_{w}\oplus\xi_{w})\pi_{F(w)\oplus G(w)}(f,g)\rangle d\mu(w) \bigg|^{2}\\&\leq \int_{\Theta}v^{2}(w)\|(\chi_{w}\oplus \xi_{w})\pi_{F(w)\oplus G(w)}S_{\chi\oplus \xi}^{-1}(f,g)\|^{2}d\mu(w)\\&\quad \quad \int_{\Theta}v^{2}(w)\|(\chi_{w}\oplus \xi_{w})\pi_{F(w)\oplus G(w)}(f,g)\|^{2}d\mu(w)\\&\leq \int_{\Theta}v^{2}(w)\|(\chi_{w}\oplus\xi_{w})\pi_{F(w)\oplus G(w)}S_{\chi\oplus \xi}^{-1}(f,g)\|^{2}d\mu(w)B_{1}\|(f,g)\|^{2}\\&=B_{1}\|(f,g)\|^{2}\int_{\Theta}v^{2}(w)\|(\chi_{w}\oplus \xi_{w})\pi_{F(w)\oplus G(w)}S_{\chi\oplus\xi}^{-1}\pi_{S_{\chi\oplus\xi}^{-1}(F(w)\oplus G(w))}(f,g)\|^{2}d\mu(w).
	\end{align*}
	Then
	\begin{equation*}
		B_{1}^{-1}\|(f,g)\|\leq \int_{\Theta}v^{2}(w)\|(\chi_{w}\oplus\xi_{w})\pi_{F(w)\oplus G(w)}S_{\chi\oplus\xi}^{-1}\pi_{S_{\chi\oplus\xi}^{-1}(F(w)\oplus G(w))}(f,g)\|^{2}d\mu(w).
	\end{equation*}
	Hence, $\delta$ is a continuous $g-$fusion frame for $U\oplus X$. Let $S_{\delta}$ be the continuous $g-$fusion frame operator for $\delta$ and take $\delta_{w}=\chi_{w}\oplus \xi_{w}$. Now, for all $(f,g)\in U\oplus X$ 
	\begin{align*}
		&S_{\delta}(f,g)\\&	=\int_{\Theta}v^{2}(w)\pi_{S_{\chi\oplus\xi}^{-1}(F(w)\oplus G(w))}(\delta_{w}\pi_{F(w)\oplus G(w)}S_{\chi\oplus\xi}^{-1})^{\ast}(\delta_{w}\pi_{F(w)\oplus G(w)}S_{\chi\oplus \xi}^{-1})\pi_{S_{\chi\oplus\xi}^{-1}(F(w)\oplus G(w))}(f,g)d\mu(w)\\&=\int_{\Theta}v^{2}(w)(\pi_{F(w)\oplus G(w)}S_{\chi\oplus\xi}^{-1}\pi_{S_{\chi\oplus\xi}^{-1}(F(w)\oplus G(w))})^{\ast}\delta_{w}^{\ast}\delta_{w}(\pi_{F(w)\oplus G(w)}S_{\chi\oplus\xi}^{-1}\pi_{S_{\chi\oplus\xi}^{-1}(F(w)\oplus G(w))})(f,g)d\mu(w)\\&=\int_{\Theta}v^{2}(w)(\pi_{F(w)\oplus G(w)}S_{\chi\oplus\xi}^{-1})^{\ast}\delta_{w}^{\ast}\delta_{w}(\pi_{F(w)\oplus G(w)}S_{\chi\oplus\xi}^{-1})(f,g)d\mu(w)\\&=\int_{\Theta}v^{2}(w)S_{\chi\oplus\xi}^{-1}\pi_{F(w)\oplus G(w)}(\chi_{w}\oplus\xi_{w})^{\ast}(\chi_{w}\oplus\xi_{w})(\pi_{F(w)\oplus G(w)}S_{\chi\oplus\xi}^{-1})(f,g)d\mu(w)\\&=S_{\chi\oplus\xi}^{-1}\bigg(\int_{\Theta}v^{2}(w)\pi_{F(w)\oplus G(w)}(\chi_{w}\oplus\xi_{w})^{\ast}(\chi_{w}\oplus\xi_{w})\pi_{F(w)\oplus G(w)}(S_{\chi\oplus\xi}^{-1}(f,g))d\mu(w)\bigg)\\&=S_{\chi\oplus\xi}^{-1}S_{\chi\oplus\xi}(S_{\chi\oplus\xi}^{-1}(f,g))\\&=S_{\chi\oplus\xi}^{-1}(f,g).
	\end{align*}
	Therefore, $S_{\delta}=S_{\chi\oplus\xi}^{-1}$. 
	
\end{proof}

	\section*{Acknowledgments}
	It is our great pleasure to thank the referee for his careful reading of the paper and for several helpful suggestions.
	
	\section*{Contributions}
	All authors contributed substantially to this paper, participated in drafting and check-
	ing the manuscript. All authors read and approved the final manuscript.
	\section*{Ethics declarations}
	
	\subsection*{Availablity of data and materials}
	Not applicable.
	\subsection*{Conflict of interest}
	The authors declare that they have no competing interests.
	\subsection*{Fundings}
	Not applicable.

\end{document}